\documentclass[12pt]{article}
\usepackage{amssymb}
\usepackage{amsfonts}
\usepackage{amsthm}
\usepackage{amsmath}
\usepackage{fullpage}
\usepackage{algorithm}
\usepackage{algorithmic}
\usepackage{mathtools}

\newtheorem{Theorem}{Theorem}

\newtheorem{Definition}{Definition}
\newtheorem{Corollary}{Corollary}

\usepackage[pdftex]{graphicx}
\graphicspath{{C:/Users/Sam/Desktop/PaperFigures/}}
\DeclareGraphicsExtensions{.pdf,.png}

\begin{document}
\title{On Contact Numbers of Finite Lattice Sphere Packings and the Maximal Coordination of Monatomic Crystals}
\author{Samuel Reid\thanks{University of Calgary, Centre for Computational and Discrete Geometry (Department of Mathematics \& Statistics), and Thangadurai Group (Department of Chemistry), Calgary, AB, Canada. $\mathsf{e-mail: smrei@ucalgary.ca}$}}
\maketitle

\begin{abstract}
We algorithmically solve the maximal contact number problem for finite
congruent lattice sphere packings in $\mathbb{R}^d$ and show that in $\mathbb{R}^3$ this problem is equivalent to determining the maximal
coordination of a monatomic crystal. 
\end{abstract}

\text{   \;\;       } \textbf{MSC 2010 Subject Classifications:} Primary 52C99, Secondary 92E10.

\section{Introduction}
There has been recent interest in the condensed matter physics and solid state chemistry community in 
constructing finite sphere packings with the maximal number of touching pairs (also known as the 
maximal contact number), as there are applications of small cluster geometry in nucleation, gelation, 
glass formation pathways, minimal clusters in colloids, and many other topics in soft matter physics, 
chemistry, and materials science. With this motivation, in 2011, N. Arkus, V. Manoharan, and 
M. Brenner at Harvard University provided various constructions of finite sphere packings with 
maximal contacts up to $n=20$ spheres  \cite{ArkusManharanBrenner}, and in 2014, M. Holmes-Cerfon at 
New York University provided constructions up to $n=18$ spheres \cite{HolmesCerfon} which improved 
the lower bounds from 2013 of K. Bezdek and S. Reid which are derived from the octahedral construction \cite{BezdekReid}. When 
$n=6,19,...,\frac{2k^{3} + k}{3}, k \in \mathbb{N}$ the lower bound corresponds to a fully
constructed octahedron (see Figure \ref{fig1} for $k=4$) and when $\frac{2k^{3} + k}{3} < n < \frac{2(k+1)^{3} + (k+1)}{3}$
the lower bound corresponds to a partially constructed octahedron. Recent work by K. Bezdek and M. Khan \cite{Khan} in 2016 has
reviewed the contact number problem for sphere packings and discussed explicit constructions for $n<12$ and 
the complexity of recognizing contact graphs; furthermore, the topic of totally separable sphere packings is emphasized,
as it is in \cite{BezdekSep}.
Regular totally separable sphere packings were enumerated in for $d=2,3,4$ in \cite{ReidSep} by S. Reid in 2015, which
have a chemical interpretation.

\begin{figure}[h!]\label{fig1}
\begin{center}
\includegraphics{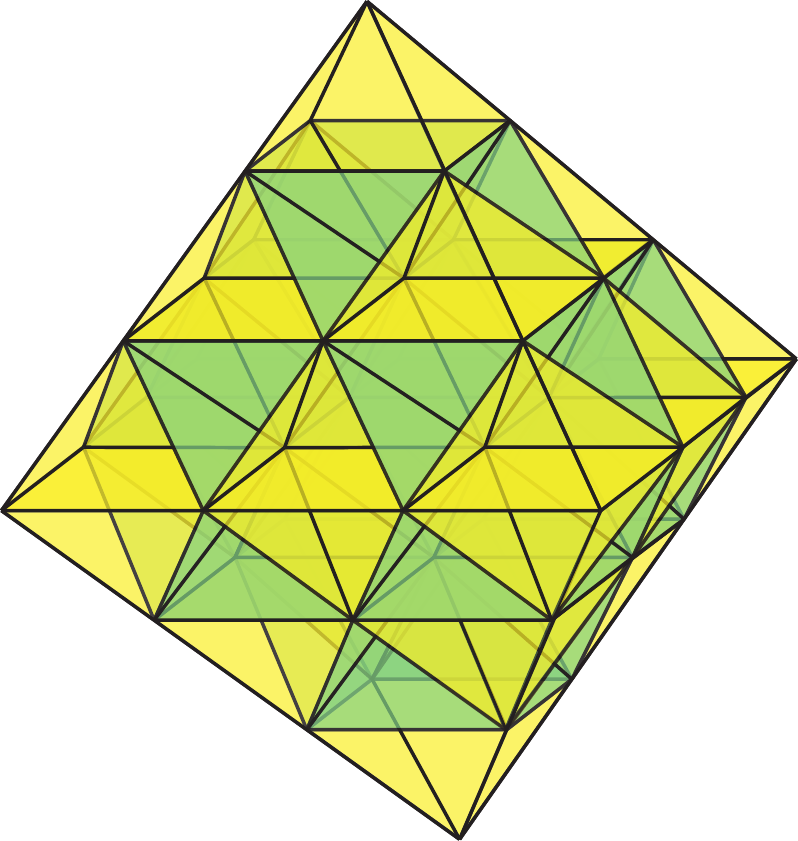}
\caption{The fourth iteration of the Octahedral Construction.}
\end{center}
\end{figure}

In light of this recent research interest at the interface of discrete geometry and materials science, we
construct a formalism for understanding all molecular geometries that translates theorems from
discrete geometry into existential bounds on the realizability and structure of chemical compounds. In particular,
we provide a comprehensive theory for understanding the structure of chemical
compounds with discrete geometry and abstract algebra using sphere packings and free $\mathbb{Z}$-modules of rank 3.
We study the case of congruent sphere packings (which are applicable for monatomic compounds) and leave the 
topic of noncongruent sphere packings (which are applicable for polyatomic compounds) for future research.

\section{Applied Discrete Geometry of Monatomic Compounds}
Traditionally, discrete geometry has been almost exclusively studied by pure mathematicians, with only
recent notice by chemists, physicists, biologists, materials scientists, computer scientists,
and other scientists. For this reason, we introduce the following definition:

\begin{Definition}[Applied Discrete Geometry]
A new interdisciplinary field of
science that describes the structure and combinatorics of matter and information with discrete geometry.
\end{Definition}

In this paper we provide some initial remarks on applied discrete geometry through crystal chemistry, e.g.,
we focus on finite congruent lattice sphere packings, rather than finite congruent sphere packings which are not based
on a lattice; these two cases correspond to monatomic crystals and amorphous monatomic compounds, respectively.

The monatomic sphere packing correspondence provides a translation between an arbitrary monatomic compound $\text{A}_{Z}$,
where $Z$ is the atomic concentration of the atom $\text{A}$, and a congruent sphere packing $\mathcal{P}_{\text{A}_{Z}}$
which encodes the relevant structural and combinatorial information about the compound.
The translation is provided by replacing atoms with spheres and replacing chemical bonds with contact points.

\begin{Definition}[Monatomic Sphere Packing Correspondence]
Let $\text{A}_{Z}$ be an arbitrary monatomic compound. Then there exists a congruent sphere
packing
$$\mathcal{P}_{\text{A}_{Z}} = \bigcup_{i=1}^{n} \left(x_{i} + r(\text{A})\mathbb{S}^2\right),$$
where $x_{i} \in \mathbb{R}^3$ is the position of the $i^{\text{th}}$ atom $\text{A}$ with radius $r(\text{A})$.
\end{Definition}

We remark that the radius $r(A)$ can be chosen to be the ionic radius of the $A$ ion with a particular coordination and charge,
with roman numeral subscripts to denote the coordination number of that atom, e.g., $r_{\text{IV}}(\text{O}^{2-})$ is
the 4-coordinated radius of the $O^{2-}$ ion.

From an arbitrary monatomic compound $\text{A}_{Z}$ there are two major cases to distinguish and treat using
the monatomic sphere packing correspondence:
\begin{enumerate}
 \item Monatomic Crystals, which correspond to congruent lattice sphere packings.
 \item Monatomic Amorphs, which correspond to congruent non-lattice sphere packings.
\end{enumerate}

A three-dimensional lattice $\Lambda$ has a basis, say $\{\omega_{1},\omega_{2},\omega_{3}\}$ where $\omega_{i} \in \mathbb{R}^3$, so that any point
on the lattice $\Lambda$ can be written as an integer linear combination of the basis elements. Then
$$\Lambda = \bigoplus_{i=1}^{3} \omega_{i}\mathbb{Z}$$
is a free $\mathbb{Z}$-module of rank 3 which is a candidate for representing the underlying structure of
a monatomic crystal $\text{A}_{Z}$ if the lattice basis is sufficiently spread out to accomodate for the size of $A$;
i.e., if $\|\omega_{i} - \omega_{j}\| \geq 2 r(A), \forall i \neq j$.

We now characterize the combinatorial features of the coordination structure associated with a monatomic compound
$\text{A}_{Z}$ in terms of the contact graph of the congruent sphere packing $\mathcal{P}_{\text{A}_{Z}}$. The
contact graph
$$G_{\mathcal{P}_{\text{A}_{Z}}} = (V,E),$$
is defined by a vertex set $V = \{x_{i} \; | \; 1 \leq i \leq Z\}$ with size $|V|=Z$ and an edge set
$$E = \{(i,j) \; | \; (x_{i} + r(A)\mathbb{S}^2) \cap (x_{j} + r(A)\mathbb{S}^2) \neq \emptyset, 1 \leq i,j \leq Z\}.$$
Then the size $|E|$ determines the number of chemical bonds $C(\text{A}_{Z})$ in the monatomic compound $\text{A}_{Z}$, and we can
use sphere packing bounds from discrete geometry to characterize this quantity in terms of $|V|=Z$.

To see how the contact graph can be used to understand the coordination structure of the crystal, we use the 
following theorem of K. Bezdek and S. Reid \cite{BezdekReid}.
\begin{Theorem}\label{bezdekreidtheorem} \text{ }
\begin{enumerate}
 \item The number of touching pairs in an arbitrary packing of $n>2$ unit balls in $\mathbb{R}^3$ is
 always less than $$6n - 0.926n^{2/3}.$$
 \item The number of touching pairs in an arbitrary lattice packing of $n>2$ unit balls in $\mathbb{R}^3$
 is always less than $$6n - \frac{3\sqrt[3]{18 \pi}}{\pi}n^{2/3} = 6n - 3.665...n^{2/3}.$$
\end{enumerate}
\end{Theorem}
\noindent We can then obtain a relevant corollary regarding the maximum number of chemical bonds in a monatomic compound,
with improved bounds in the case of a monatomic crystal.
\begin{Corollary}\label{bezdekreidcorollary}
For any atom $\text{A}$, the number of chemical bonds in an amorphous $\text{A}_{Z}$ compound is always less than 
$6Z - 0.926Z^{2/3}$, and the number of chemical bonds in a crystal $\text{A}_{Z}$ compound is always less than $6Z - 3.665...Z^{2/3}$.
\end{Corollary}

\section{The Maximal Lattice Contact Number Algorithm}
Let $\mathcal{P}_{\Lambda}$ be a packing of $n$ congruent $(d-1)$-spheres in $d$-space placed over the lattice $\Lambda$, i.e., $\mathbb{S}^{d-1} \hookrightarrow \Lambda \subset \mathbb{R}^d$.
We define the maximal contact number of $\mathcal{P}_{\Lambda}$ by
$$C_{d}(\mathcal{P}_{\Lambda},n) = \max_{\mathcal{P} \subset \mathcal{P}_{\Lambda}} \left\{ |E(\mathcal{P})| \; \big| \; |V(\mathcal{P})| = n\right\}.$$
In terms of this terminology we can restate Theorem \ref{bezdekreidtheorem}.2 as
$$C_{3}(\mathcal{P}_{\Lambda},n) < 6n - \frac{3\sqrt[3]{18 \pi}}{\pi}n^{2/3} = 6n - 3.665...n^{2/3}, \forall n > 2.$$

We now prove the following theorem which characterizes the vertex set of a lattice packing which has the maximal contact
number.

\begin{Theorem}\label{vertexcharacterization}
If $C_{d}(\mathcal{P}_{\Lambda},n) = |E(\mathcal{P})|$ then
$$V(\mathcal{P})\subseteq \left\{\sum_{i=1}^{d} \lambda_{i} \omega_{i} \; \big| \; 0 \leq \lambda_{i} \leq \lceil n/d \rceil, \forall 1 \leq i \leq d\right\},$$
where $\Lambda = \displaystyle\bigoplus_{i=1}^{d} \omega_{i} \mathbb{Z}$ is a free $\mathbb{Z}$-module of rank $d$
and $\omega_{i} \in \mathbb{R}^d, \forall 1 \leq i \leq d, \| \omega_{i} - \omega_{j} \| \geq 2, \forall i \neq j$.
\end{Theorem}
\begin{proof}
Let $n \in \mathbb{N}$ and assume that $C_{d}(\mathcal{P}_{\Lambda},n) = |E(\mathcal{P})|$. Observe that $\lceil \frac{n}{d} \rceil$ is an upper bound on the necessary
number of elements in every submodule $\omega_{i} \mathbb{Z}, 1 \leq i \leq d$ required to construct
$\mathcal{P} \subset \mathcal{P}_{\Lambda}$ with $C_{d}(\mathcal{P}_{\Lambda},n) = E(\mathcal{P})$. For, if a submodule
$\omega_{j}\mathbb{Z}$ contained more than $\lceil \frac{n}{d} \rceil$ elements, the packing could be rearranged to
decrease the number of elements of $\omega_{j}\mathbb{Z}$ and
increase the number of elements of a distinct submodule $\omega_{k}\mathbb{Z}$ in a way which increases the contact number,
contradicting that $C_{d}(\mathcal{P}_{\Lambda},n) = |E(\mathcal{P})|$.
\end{proof}

\begin{Corollary}
If $C_{3}(\mathcal{P}_{\Lambda},n) = |E(\mathcal{P})|$ then
$$V(\mathcal{P})\subseteq \left\{\lambda_{1}\omega_{1} + \lambda_{2}\omega_{2} + \lambda_{3}\omega_{3} \; \big| \; 0 \leq \lambda_{i} \leq \lceil n/3 \rceil, \forall 1 \leq i \leq 3\right\},$$
where $\Lambda = \omega_{1}\mathbb{Z} \oplus \omega_{2}\mathbb{Z} \oplus \omega_{3}\mathbb{Z}$ is
a free $\mathbb{Z}$-module of rank 3 with $\| \omega_{i} - \omega_{j} \| \geq 2, \forall i \neq j$,
$\omega_{1},\omega_{2},\omega_{3} \in \mathbb{R}^3$.
\end{Corollary}

We now use the above corollary to write an algorithm for obtaining the value of $C_{3}(\mathcal{P}_{\Lambda},n)$
over any three-dimensional lattice $\Lambda$. Chemically, this algorithm tells us the maximum number 
of chemical bonds between any crystalline formation of $n$ atoms of the same size.

\begin{algorithm}
\caption{Maximal Coordination of Monatomic Crystals}
\begin{algorithmic}
\REQUIRE $\Lambda = \omega_{1}\mathbb{Z} \oplus \omega_{2}\mathbb{Z} \oplus \omega_{3}\mathbb{Z}$, where $\| \omega_{i} - \omega_{j} \| \geq 2r(A), \forall i \neq j, \omega_{1},\omega_{2},\omega_{3} \in \mathbb{R}^3$, $n \in \mathbb{N}$.
\STATE Set $\mathcal{P}_{\Lambda} = \varnothing$ and $k=\lceil n/3 \rceil$.
\FORALL{$\lambda_{1} \in \{0,1,...,k\}$}
\FORALL{$\lambda_{2} \in \{0,1,...,k\}$}
\FORALL{$\lambda_{3} \in \{0,1,...,k\}$}
\STATE Set $\mathcal{P}(\lambda_{1},\lambda_{2},\lambda_{3}) = \lambda_{1}\omega_{1} + \lambda_{2}\omega_{2} + \lambda_{3}\omega_{3} + r(A)\mathbb{S}^2$.
\STATE Set $\mathcal{P}_{\Lambda} = \mathcal{P}_{\Lambda} \cup \mathcal{P}(\lambda_{1},\lambda_{2},\lambda_{3})$.
\ENDFOR
\ENDFOR
\ENDFOR
\STATE Set $C_{3}(\mathcal{P}_{\Lambda},n)=n$.
\FORALL{$S \subseteq \mathfrak{P}(\mathcal{P}_{\Lambda})$}
\IF{$|S|=n$}
\IF{$|E(S)| > C_{3}(\mathcal{P}_{\Lambda},n)$}
\STATE Set $C_{3}(\mathcal{P}_{\Lambda},n)=|E(S)|$.
\ENDIF
\ENDIF
\ENDFOR
\RETURN $C_{3}(\mathcal{P}_{\Lambda},n)$.
\end{algorithmic}
\end{algorithm}

\pagebreak

\end{document}